\documentclass[a4paper, 12pt]{amsart}
\usepackage{cite}
\usepackage{amsmath}
\usepackage{amsthm}
\usepackage[all]{xy}
\usepackage{amssymb}
\usepackage{mathrsfs}
\usepackage{etoolbox}
\usepackage{fullpage}

\newtheorem{thm}{Theorem}
[section]

\newtheorem{lem}[thm]{Lemma}
\newtheorem{prop}[thm]{Proposition}

\theoremstyle{definition}

\newtheorem{fct}[thm]{Fact}

\numberwithin{equation}{subsection}

\newcommand{\gen}[1]{\left\langle#1\right\rangle}
\usepackage{color}
\begin{document}

\title{Algebraic independence of generic Painlev\'e  transcendents: $P_{III}$ and $P_{VI}$}
\author{Joel Nagloo* \\Bronx Community College, CUNY}
\date{\today}
\address{Department of Mathematics and Computer Science\\Bronx Community College, CUNY\\Bronx, NY 10453, United States.}
\email{joel.nagloo@bcc.cuny.edu}
\pagestyle{plain}
\thanks{*Research supported by NSF grant DMS-1700336 and PSC-CUNY grant \#60456-00.}
\begin{abstract}
We prove that if $y''=f(y,y',t,\alpha,\beta,\gamma,\delta)$ is a generic Painlev\'e equation from the class $III$ and $VI$, and if $y_{1},...,y_{n}$ are distinct solutions, then $tr.deg_{\mathbb{C}(t)}\mathbb{C}(t)(y_1,y'_1,\ldots,y_n,y'_n)=2n$, that is $y_1,y'_1,\ldots,$ $y_n,y'_n$ are algebraically independent over $\mathbb{C}(t)$. This improves the results obtained by the author and Pillay and completely proves the algebraic independence conjecture for the generic Painlev\'e  transcendents. In the process, we also prove that any three distinct solutions of a Riccati equation are algebraic independent over $\mathbb{C}(t)$, provided that there are no solutions in the algebraic closure of $\mathbb{C}(t)$. This answers a very natural question in the theory.
\end{abstract}
\maketitle

\section{Introduction}
The algebraic independence conjecture for the generic Painlev\'e equations $P_{I}-P_{VI}$ \cite{NagPil} states that if $y_1,\ldots,y_n$ are distinct solutions of one of the generic equations viewed as meromorphic functions on some disc $D\subset\mathbb{C}$ and if we work in the differential field $\mathcal{F}$ of meromorphic functions on $D$, then $tr.deg_{\mathbb{C}(t)}\mathbb{C}(t)(y_1,y'_1,\ldots,y_n,y'_n)=2n$, that is $y_1,y'_1,\ldots,y_n,y'_n$ are algebraically independent over $\mathbb{C}(t)$. Recall that the Painlev\'e equations are
\begin{equation*}
\begin{split}
P_{I}:\;\;\;\;\; & y'' =6y^2+t \\
P_{II}(\alpha):\;\;\;\;\; & y'' =2y^3+ty+\alpha \\
P_{III}(\alpha,\beta,\gamma,\delta):\;\;\;\;\; & y'' =\frac{1}{y}(y')^2 -\frac{1}{t}y'+\frac{1}{t}(\alpha y^2+\beta)+\gamma y^3+\frac{\delta}{y} \\
P_{IV}(\alpha,\beta):\;\;\;\;\; & y'' =\frac{1}{2y}(y')^2+\frac{3}{2}y^3+4ty^2+2(t^2-\alpha)y+\frac{\beta}{y} \\
P_{V}(\alpha,\beta,\gamma,\delta):\;\;\;\;\; & y'' =\left(\frac{1}{2y}+\frac{1}{y-1}\right)(y')^2-\frac{1}{t}y'+\frac{(y-1)^2}{t^2}\left(\alpha y+\frac{\beta}{y}\right)+\gamma\frac{y}{t}\\
& +\delta\frac{y(y+1)}{y-1}\\
\end{split}
\end{equation*} 
\begin{equation*}
\begin{split}
P_{VI}(\alpha,\beta,\gamma,\delta):\;\;\;\;\; & y'' =\frac{1}{2}\left(\frac{1}{y}+\frac{1}{y+1}+\frac{1}{y-t}\right)(y')^2-\left(\frac{1}{t}+\frac{1}{t-1}+\frac{1}{y-t}\right)y'\\
 & +\frac{y(y-1)(y-t)}{t^2(t-1)^2}\left(\alpha+\beta\frac{t}{y^2}+\gamma\frac{t-1}{(y-1)^2}+\delta\frac{t(t-1)}{(y-t)^2}\right),
\end{split}
\end{equation*}
where $\alpha$, $\beta$, $\gamma$, $\delta\in\mathbb{C}$. By generic, here we mean that the relevant complex parameters $\alpha$, $\beta$,... are algebraically independent over $\mathbb{Q}$, while the single equation $P_{I}$ is considered generic in its own class.
\par For $P_{I}$, the conjecture was shown to be true by Nishioka \cite{Nishioka} using techniques from the study of differential function fields in one variable. On the other hand using model theory, Nagloo and Pillay \cite{NagPil1} proved that the conjecture is also true in the case of the generic $P_{II}$, $P_{IV}$ and  $P_{V}$. They moreover obtained a weaker result for $P_{III}$ and $P_{VI}$: given distinct solutions $y_{1},..,y_{k}$ of generic $P_{III}$ (resp. $P_{VI}$) such that $tr.deg_{\mathbb{C}(t)}\mathbb{C}(t)(y_1,y'_1,\ldots,y_k,y'_k)= 2k$, then for all other solutions $y$, except for at most $k$ (resp. $11k$), $tr.deg_{\mathbb{C}(t)}\mathbb{C}(t)(y_1,y'_1,\ldots,y_k,y'_k,y,y')= 2(k+1)$.
\par In this paper, we show that the conjecture also holds in the case of the generic $P_{III}$ and $P_{VI}$, hence settling the question entirely. The methods used in the proof are quite general and can be used to reprove the conjecture for the generic $P_{II}$, $P_{IV}$ and  $P_{V}$ (although not done here). They combine both the differential algebraic methods used by Nishioka in \cite{Nishioka} and the model theoretic techniques from \cite{NagPil1}.
\par For simplicity, let us use the second Painlev\'e family to illustrate how the proof goes. Fix $\alpha_0\in\mathbb{C}$ a transcendental constant. First recall that $P_{II}(\alpha_0)$ is strongly minimal: if $y$ is a solution which satisfies a first order algebraic differential equation over a differential field extension $F$ of $\mathbb{C}(t)$, then $y$ is algebraic over $F$. On the other hand $P_{II}(1/2)$ is not strongly minimal. Indeed, it is well known that any solution of the Riccati equation $y'=y^2+t/2$ is also a solution of $P_{II}(1/2)$. Secondly, it is also known that $P_{II}(\alpha_0)$ is geometrically trivial, namely if every pair of distinct solutions and derivatives are algebraically independent over some differential field extension $F$ of $\mathbb{C}(t)$, then every finite collection of distinct solutions and derivatives are algebraically independent over $F$. 
\par The (new) proof of the algebraic independence conjecture for $P_{II}(\alpha_0)$ has two main steps: (i) showing that any two solutions of the above Riccati equation are algebraically independent over $\mathbb{C}(t)$; and (ii) using geometric triviality (assuming the conjecture is false), the genericity of the parameter $\alpha_0$ and specialization, to obtain algebraic dependency between some pair of solutions of the Riccati equation, contradicting the result in the first step. As already mentioned, this strategy is very general. One needs to simply observed that all the Painlev\'e families have properties similar to the above: for generic parameters, the equations are strongly minimal and geometrically trivial, while for some special values of the parameters Riccati equations exist.
\par The paper is organized as follows. In Section 2 we prove an algebraic independence result for the Riccati equations. Namely, we show that any three distinct solutions of are algebraic independent over $\mathbb{C}(t)$, provided that there are no solutions in the algebraic closure of $\mathbb{C}(t)$. This answers a very natural question in the theory. Section 3 is where the main results is proved. We use the Riccati equations attached to the Painlev\'e families and the results in Section 2 to prove the algebraic independence conjecture for the generic $P_{III}$ and $P_{VI}$.


\section{An independence result for the Riccati equations}

We assume that the reader is familiar with the basics of differential algebraic geometry and the model theory of differentially closed fields as presented in, say, Marker \cite{Marker}. The paper \cite{NagPil}, which the current work is a continuation of, also contains a very good summary of the main notion used here.
 
We fix once and for all $\mathcal{U}$, a saturated model of $DCF_0$, the theory of differentially closed fields of characteristic $0$ with a single derivation in the language $L_{\partial}= (+,-,\cdot,0,1,\partial)$ of differential rings. We will assume that $\mathbb{C}$, the field of complex numbers, is the field of constants of $\mathcal{U}$ and that $t$ denote an element of $\mathcal{U}$ with the property that $\partial(t)=1$.

Recall that the Riccati equation (over $\mathbb{C}(t)$) is given by 
\[y'=ay^2+by+c,\] 
where $a,b,c\in\mathbb{C}(t)$ and $a\neq0$. Throughout, we denote the set it defines by $Ric(a,b,c)$ and we will assume that $Ric(a,b,c)$ has no elements in $\mathbb{C}(t)^{alg}$. This is of course the case for all the Riccati equations that appear in the study of the Painlev\'e families. Let us recall some of the well known properties of $Ric(a,b,c)$. The first two can easily be verified:
\begin{fct}
$Ric(a,b,c)$ is strongly minimal.
\end{fct}
\begin{fct}\label{RicSimple}
The map $y\mapsto ay$ is a definable bijection between $Ric(a,b,c)$ and $Ric(1,r,s)$ where $r=b+\frac{a'}{a}$ and $s=ac$. So in particular $Ric(1,r,s)$ has no solutions in $\mathbb{C}(t)^{alg}$.
\end{fct}
Unlike the generic Painlev\'e equations, $Ric(a,b,c)$ is not geometrically trivial. It is well known that for any $y_1,y_2,y_3,y_4\in Ric(a,b,c)$, one has that $tr.deg_{\mathbb{C}(t)}\mathbb{C}(t)(y_1,y_2,y_3,y_4)<4$. This is explained by the existence of a superposition law; namely, given three distinct elements $y_1,y_2,y_3\in Ric(a,b,c)$, then for any other $y\in Ric(a,b,c)$, there is $\alpha\in\mathbb{C}$ such that 
\[y=\frac{y_2(y_3-y_1)+\alpha y_1(y_2-y_3)}{\alpha(y_3-y_1)+\alpha(y_2-y_3)}.\] This is equivalent to the fact that the cross ratio of $y_1$, $y_2$, $y_3$ and $y_4$ is constant. On the other hand, one has the following natural question: given distinct elements $y_1,y_2,y_3\in Ric(a,b,c)$, is $tr.deg_{\mathbb{C}(t)}\mathbb{C}(t)(y_1,y_2,y_3)=3$? We will show that this is indeed the case. First though, we prove that any two distinct solutions are algebraically independent over $\mathbb{C}(t)$. As mentioned in the introduction, we will use this in the next section to prove the algebraic independence conjecture for the Painlev\'e equations. Moreover, this is also needed to prove the general transcendence statement for $y_1,y_2,y_3\in Ric(a,b,c)$.\\

Before we proceed recall that for a field $L$,  $L\left(\left(X\right)\right)$ denotes the field of formal Laurent series in variable $X$, while $L\gen{\gen{X}}$ denotes the field of formal Puiseux series, i.e. the field $\bigcup_{d\in{\mathbb{N}}}L\left(\left(X^{1/d}\right)\right)$. It is well know that if $L$ is an algebraically closed field of characteristic zero, then so is $L\gen{\gen{X}}$ (cf. \cite{Eisenbud}). 

\begin{prop}\label{Mainprop}
Let $y_1,y_2\in Ric(a,b,c)$ be distinct. Then $tr.deg_{\mathbb{C}(t)}\mathbb{C}(t)(y_1,y_2)=2$, that is $y_1$ and $y_1$ are algebraically independent over $\mathbb{C}(t)$.
\end{prop}
\begin{lem}\label{Mainlem}
Let $r,s\in\mathbb{C}(t)$ and suppose that $K$ is a differential extension of $\mathbb{C}(t)$ such that $Ric(1,r,s)$ has no solutions in $K^{alg}$. Let $y,z\in Ric(1,r,s)$ be distinct. If $z\in K(y)^{alg}$, then $z\in K^{alg}(y)$.
 \end{lem}
\begin{proof}
Let $y$ and $z$ be two distinct elements of $ Ric(1,r,s)$ and suppose $z\in K(y)^{alg}$. We write $L=K^{alg}$. Now since $z\in L(y)^{alg}$, we can look at expansions in a local parameter $\tau$ at $\beta\in F$ given by
\[y=\beta+\tau^e\;\;\;\;\;\;\;\;\;\;\;\;\;z=\sum_{i=r}^{\infty}a_i\tau^i\;\;\;\;(a_r\neq0)\]
with $e$ the ramification exponent. In other words, for any $\beta\in L$, $z$ can be seen as an element of $L\gen{\gen{y-\beta}}$ and so there is $e\in\mathbb{N}$ such that $z\in L\left(\left((y-\beta)^{1/e}\right)\right)$. All we have to show is that for every choice of $\beta\in L$, the ramification exponent is $1$ ($e=1$). Using the genus formula of Hurwitz (cf. \cite{Lang} Theorem 6.1), it then follows that the degree of the extension $L(y,z)$ over $L(y)$  is $1$ and we are done.
\par Differentiating we have
\begin{eqnarray}
 \notag e\tau^{e-1}\tau' &=& y'-\beta'\\ \notag
	&=& y^2+ry+s-\beta'\\ \notag 
	&=& (\beta+\tau^e)^2+r(\beta+\tau^e)+s-\beta'\\
	&=& \beta^2+r\beta+s-\beta'+(2\beta+r)\tau^e+\tau^{2e}\label{eqn0}
\end{eqnarray}
Letting $\gamma=\beta^2+r\beta+s-\beta'$, we have that $\gamma\neq0$, since $Ric(1,r,s)$ has no element in $L=K^{alg}$. Hence from equation \ref{eqn0}, $\gamma\neq0$ implies that $\tau'=e^{-1}\gamma\tau^{1-e}+A$ where 
\[ 
A=e^{-1}\tau^{1-e}\left((2\beta+r)\tau^e+\tau^{2e}\right)
\]
and from this we get that
\begin{eqnarray*}
(\tau^i)'&=& i\tau^{i-1}\tau'\\
	&=& i\tau^{i-1}e^{-1}\gamma\tau^{1-e}+A_1\\
	&=& \frac{i\gamma}{e}\tau^{i-e}+A_1
\end{eqnarray*}
Now using the equations
\[z'=z^2+rz+s\;\;\;\;\;\;\;\;\text{and}\;\;\;\;\;z=\sum_{i=r}^{\infty}a_i\tau^i\]
we have 
\begin{equation}\label{main}
\sum_{i=r}^{\infty}a_i \frac{i\gamma}{e}\tau^{i-e}+\cdots=\sum_{i,j}a_ia_j\tau^{i+j}+\cdots.
\end{equation}
Using this we prove a couple of claims. One should note that in the calculations below, when using \ref{main}, we will carefully choose powers of $\tau$ in such a way that when comparing coefficients on both sides of the equation, only those shown above will play a role.\\
{\bf Claim 1:} If $e>1$ then $r<0$.\\
{\em Proof:} Let $e>1$ and assume $r\geq0$. If for all $l\in\{r,r+1,\ldots\}$ $e\mid l$, then it is not hard to see that the ramification exponent must be one (i.e $e=1$), a contradiction. So choose $l\in\{r,r+1,\ldots\}$ least such that $e\nmid l$ and $a_l\neq0$. First, one should note that since $l-e<l$ and $e\nmid l-e$, $a_{l-e}=0$, so that in what follows, as we will look at the coefficient of $\tau^{l-e}$, one does not need to worry about the other coefficients in \ref{main} (especially the linear part of $z^2+rz+s$).
\par The coefficient of $\tau^{l-e}$ on the LHS of \ref{main} is
\[a_l\frac{l\gamma}{e}\neq0.\]
This implies that the coefficient on the RHS of $\tau^{i+j}$ for some $i,j\geq r$ with $i+j=l-e$ must be non-zero. However for any such, since $i+j<l$ and $e\nmid i+j$, we have that $e$ does not divide at least one them, say $i<l$. But then $a_i=0$ (as $l$ was chosen to be the least with this property) and so $a_ia_j=0$, a contradiction.\\
{\bf Claim 2:} The case $e>1$ and $r<0$ leads to a contradiction.\\
{\em Proof:} So this time suppose $e>1$ and  $r<0$. From the least powers of $\tau$ in $\ref{main}$ we have $r-e=2r$, that is $r=-e$, and from the coefficients of $\tau^{r-e}$ we get
\[a_r\frac{r\gamma}{e}=a_r^2\]
so that $\gamma=-a_r$, since $a_r\neq0$.\\
So again choose $l\in\{r,r+1,\ldots\}$ least such that $e\nmid l$ and $a_l\neq0$. The coefficient of $\tau^{l-e}=\tau^{l+r}$ on the LHS of \ref{main} is 
\[a_l\frac{l\gamma}{e}\neq0.\]
On the RHS we see that the coefficient of $\tau^{l-e}=\tau^{l+r}$ should be $2a_ra_l$. (Indeed, $e\nmid i+j$ means that $e$ does not divide at least one of them, say $i$. Then $e\nmid i$, $a_i\neq0$ means either $i=l$ or $i>l$. But $i>l$ and $i+j=r+l$ implies that $j<r$ a contradiction. So $i=l$ and hence $j=r$).\\
Hence
\[a_l\frac{l\gamma}{e}=2a_ra_l=-2\gamma a_l\]
and we see that $l=-2e$, contradicting $e\nmid l$ and we are done\\

\par Hence $e=1$, and since $\beta$ was arbitrary, the ramification exponent at every $\beta\in L$ is 1. So $z\in L(y)$.
\end{proof}
\noindent{\em Proof of Proposition \ref{Mainprop}} Using Fact \ref{RicSimple}, it suffices to prove the result for $Ric(1,r,s)$ with $r,s\in \mathbb{C}(t)$. Of course, in our case we have $r=b+\frac{a'}{a}$ and $s=ac$. So let $y_1,y_2\in Ric(1,r,s)$ be distinct and suppose $tr.deg_{\mathbb{C}(t)}\mathbb{C}(t)(y_1,y_2)=1$. Then $y_2\in \mathbb{C}(t)(y_1)^{alg}$ and so by Lemma \ref{Mainlem} we have $y_2\in \mathbb{C}(t)^{alg}(y_1)$, that is \[y_2=\frac{p(y_1)}{q(y_1)}\] for some $p,q\in \mathbb{C}(t)^{alg}[u]$.
\par The $L_{\mathbb{C}(t)^{alg}}$-formula $\exists u\left(u=\frac{p(y_1)}{q(y_1)}\right)$ is in $tp(y_1/\mathbb{C}(t)^{alg})$ and since $Ric(1,r,s)$ is strongly minimal and has not solutions in $\mathbb{C}(t)^{alg}$, we have that $\forall v\exists u\left(u=\frac{p(v)}{q(v)}\right)$ is true in $\mathcal{U}$. Here we take $u$ and $v$ to range over elements of $Ric(1,r,s)$. So let $y_3,y_4\in Ric(1,r,s)$ distinct from $y_1$ and $y_2$ be such that \[y_4=\frac{p(y_3)}{q(y_3)}\].
But the cross ratio of $y_1$, $y_2$, $y_3$ and $y_4$ is constant. So for some $\alpha\in\mathbb{C}$
\begin{eqnarray*}
\alpha &=& \frac{(y_4-y_2)(y_3-y_1)}{(y_4-y_1)(y_3-y_2)}\\
           &=& \frac{(\frac{p(y_3)}{q(y_3)}-\frac{p(y_1)}{q(y_1)})(y_3-y_1)}{(\frac{p(y_3)}{q(y_3)}-y_1)(y_3-\frac{p(y_1)}{q(y_1)})}\\
           &=& \frac{\left(p(y_3)q(y_1)-p(y_1)q(y_3)\right)\left(y_3-y_1\right)}{\left(p(y_3)-y_1q(y_3)\right)\left(y_3q(y_1)-p(y_1)\right)}.
\end{eqnarray*}
By clearing denominators we get \[\alpha\left(p(y_3)-y_1q(y_3)\right)\left(y_3q(y_1)-p(y_1)\right)-\left(p(y_3)q(y_1)-p(y_1)q(y_3)\right)\left(y_3-y_1\right)=0.\]
The polynomial $F(u,v)=\alpha\left(p(u)-vq(u)\right)\left(uq(v)-p(v)\right)-\left(p(u)q(v)-p(v)q(u)\right)\left(u-v\right)$ is not identically zero (in $\mathbb{C}(t)^{alg}[u,v]$)  since otherwise it would mean that the cross ratio of any four solutions $u,v,\frac{p(u)}{q(u)},\frac{p(v)}{q(v)}$ equals the same constant $\alpha$. So we have that $y_3\in \mathbb{C}(t)(y_1)^{alg}$ and by Lemma \ref{Mainlem} $y_3\in \mathbb{C}(t)^{alg}(y_1)$.
\par Consequently, $y_2,y_3,y_4\in \mathbb{C}(t)^{alg}(y_1)$.  Say $y_i=f_i(y_1)$ for $i=2,3,4$, where $f_i\in \mathbb{C}(t)^{alg}(u)$. We apply the cross ratio one more time. But first, notice that $\exists u_2, u_3,u_4\bigwedge_i\left(u_i=f_i(y_1)\right)$ is in $tp(y_1/\mathbb{C}(t)^{alg})$ and so as before we have that $\forall u_1\exists u_2, u_3,u_4\bigwedge_i\left(u_i=f_i(u_1)\right)$ is true in $\mathcal{U}$. So
\begin{eqnarray*}
\alpha &=& \frac{(y_4-y_2)(y_3-y_1)}{(y_4-y_1)(y_3-y_2)}\\
           &=& \frac{(f_4(y_1)-f_2(y_1))(f_3(y_1)-y_1)}{(f_3(y_1)-y_1)(f_3(y_1)-f_2(y_1))}
\end{eqnarray*}
By clearing denominators we obtain this time a polynomial $F\in \mathbb{C}(t)^{alg}[u]$ not identically zero (since again any four solutions $u,f_2(u),f_3(u),f_4(u)$ cannot have the same cross ratio $\alpha$) and such that $F(y_1)=0$. In other words $y_1\in \mathbb{C}(t)^{alg}$ a contradiction.
\qed

Using Lemma \ref{Mainlem} and the same reasoning as in the proof of Proposition \ref{Mainprop}, we get desired result for three distinct elements of $Ric(a,b,c)$.

\begin{prop}\label{Mainquestion}
Let $y_0,y_1,y_1\in Ric(a,b,c)$ be distinct. Then $tr.deg_{\mathbb{C}(t)}\mathbb{C}(t)(y_0,y_1,y_2)=3$, that is $y_0$, $y_1$ and $y_2$ are algebraically independent over $\mathbb{C}(t)$.
\end{prop}
\begin{proof}
As before, it suffices to prove the result for $Ric(1,r,s)$ with $r,s\in \mathbb{C}(t)$ (see Fact \ref{RicSimple}). So let $y_0,y_1,y_1\in Ric(1,r,s)$ be distinct and suppose $tr.deg_{\mathbb{C}(t)}\mathbb{C}(t)(y_0,y_1,y_2)<3$. By Proposition \ref{Mainprop} we must have that $tr.deg_{\mathbb{C}(t)}\mathbb{C}(t)(y_0,y_1,y_2)=2$ and so without lost of generality we assume $y_2\in \mathbb{C}(t)(y_0,y_1)^{alg}$. 
\par Notice though that by Proposition \ref{Mainprop}, we have that $ Ric(1,r,s)$ has no solutions in $\mathbb{C}(t)(y_0)^{alg}$. We think of $y_3$ as being in $\mathbb{C}(t)(y_0)(y_1)^{alg}$ and so can apply Lemma \ref{Mainlem} to get that have $y_2\in \mathbb{C}(t)(y_0)^{alg}(y_1)$. Say $y_2=f(y_1)$, where $f\in\mathbb{C}(t)(y_0)^{alg}(u)$. Note that since $Ric(1,r,s)$ is strongly minimal and has no solution in $\mathbb{C}(t)(y_0)^{alg}$, we have that the statement $\forall u\exists v \left(v=f(u)\right)$ is true in $\mathcal{U}$. We can use $f$ to rewrite the superposition law in terms of $y_0$ and $y_1$ only. Indeed, any other solution of $y\in Ric(1,r,s)$, different from $y_0$, $y_1$ and $y_2$, are of the form
\begin{eqnarray*}
y&=&\frac{y_1(y_2-y_0)+\alpha y_0(y_1-y_2)}{\alpha(y_2-y_0)+\alpha(y_1-y_2)}\\
 &=& \frac{y_1(f(y_1)-y_0)+\alpha y_0(y_1-f(y_1))}{\alpha(f(y_1)-y_0)+\alpha(y_1-f(y_1))}
\end{eqnarray*}
for some $\alpha\in\mathbb{C}$. We write the last expression for the superposition law as $y=P(\alpha,y_0,y_1,f(y_1))$.
\par Let $\beta,\gamma\in\mathbb{C}$ be distinct and let $y_3=P(\beta,y_0,y_1,f(y_1))$ and $y_4=P(\gamma,y_0,y_1,f(y_1))$. Using the fact that the cross ratio of $y_1$, $y_2$, $y_3$ and $y_4$ is constant, we have that for some $\delta\in\mathbb{C}$
\begin{eqnarray*}
\delta &=& \frac{(y_4-y_2)(y_3-y_1)}{(y_4-y_1)(y_3-y_2)}\\
           &=& \frac{(P(\gamma,y_0,y_1,f(y_1))-f(y_1))(P(\beta,y_0,y_1,f(y_1))-y_1)}{(P(\gamma,y_0,y_1,f(y_1))-y_1)(P(\beta,y_0,y_1,f(y_1))-f(y_1))}.
\end{eqnarray*}
But the cross ratio of any four solutions $u$, $f(u)$, $P(\beta,y_0,u,f(u))$ and $P(\gamma,y_0,u,f(u))$ (all distinct from $y_0$) cannot be the same constant $\delta$ and hence we have an algebraic relation between $y_0$ and $y_1$ over $\mathbb{C}(t)^{alg}$ a contradiction.
\end{proof}

\section{Algebraic independence of generic $P_{III}$ and $P_{VI}$}
We are now ready to prove the main results in this paper. As in the previous section, $K$ still denotes the differential field $\mathbb{C}(t)$.

\subsection{The family $P_{III}$.}

The family $P_{III}$ is a $4$-parameter family: $P_{III}(\alpha,\beta,\gamma,\delta)$, $\alpha,\beta,\gamma,\delta\in\mathbb{C}$ is given by the following 
\begin{equation*}
y''=\frac{1}{y}(y')^2-\frac{1}{t}y'+\frac{1}{t}(\alpha y^2+\beta)+\gamma y^3+\frac{\delta}{y}.
\end{equation*}

Moreover, as explained in \cite{NagPil}, to study $P_{III}(\alpha,\beta,\gamma,\delta)$ when $\alpha,\beta,\gamma,\delta$ are algebraically independent over $\mathbb{Q}$, it is sufficient to work with a rewriting of the equation as a $2$-parameter Hamiltonian system: 

\[S_{III}({v_{1},v_{2}})\left\{
\begin{array}{rll}
y'&=&\frac{1}{t}(2y^2x-y^2+v_1y+t)\\
x'&=&\frac{1}{t}(-2yx^2+2yx-v_1x+\frac{1}{2}(v_1+v_2))
\end{array}\right.\]
where $\alpha = 4v_{2}$, $\beta = -4(v_{1}-1)$ and where $\gamma$ and $\delta$ are replaced by $4$ and $-4$ respectively  (also see \cite{Okam4} and \cite{Umemura3}). This follows since the two equations (or rather the sets they define) are in definable bijection. Note that in particular $\alpha$ and $\beta$ are algebraically independent over $\mathbb{Q}$ if and only if $v_{1}$ and $v_{2}$ are.

Not only does the Hamiltonian form of the equation reveals some of its symmetries, it also allows us to easily identify some of the family's Riccati subvarieties. We are mainly interested in the following subvarieties:
\begin{fct}{( \cite{Murata2}, \cite{Umemura3})}\\

(1) Suppose that $v_{1}=v_{2}$, then the Riccati variety $ric(v_1)$ (= $Ric(\frac{1}{t},\frac{v_{1}}{t},1)\times\{1\}$) defined by
\[
y'=\frac{1}{t}(y^2+v_{1}y+t)\;\;\;\;\;\text{ and }\;\;\;\;\;x=1
\]\vspace{0.05 in}
is a differential subvariety of $S_{III}(v_{1},v_{2})$. Furthermore, $ric(v_1)$ has no solution in $K^{alg}$.\\
(2) Suppose that $v_{1}-v_{2}\in 2\mathbb{Z}$, then there is an order one $\mathbb{C}(t)$-differential subvariety $\mathcal{R}(v_1,v_2)$ of $S_{III}(v_{1},v_{2})$. Furthermore, $\mathcal{R}(v_1,v_2)$ has no solutions in $K^{alg}$.

\end{fct}
It is well known that there are $K$-definable bijections called the Backlund transformations of $S_{III}(v_{1},v_{2})$ and that the subvarieties $\mathcal{R}(v_1,v_2)$ are obtained from $ric(v_1)$ using those transformations (c.f. \cite{Okam4}). In particular, note that for any $v_1\in\mathbb{C}$ one has that $ric(v_1)=\mathcal{R}(v_1,v_1)$.

\begin{prop}\label{Secondprop}
Suppose $v_{1}-v_{2}\in 2\mathbb{Z}$ and let $y,z\in \mathcal{R}(v_1,v_2)$ be distinct. Then $tr.deg_{K}K(y,z)=2$, that is $y$ and $z$ are algebraically independent over $K=\mathbb{C}(t)$.
\end{prop}
\begin{proof}
This really follows from Proposition \ref{Mainprop}. Indeed, from the latter we have that if $y,z\in Ric(v_1)$ are distinct, then $tr.deg_{K}K(y,z)=2$. Using the Backlund transformations (especially the fact that they are $\mathbb{C}(t)$-definable bijections), we have that the same must be true of distinct solutions of $\mathcal{R}(v_1,v_2)$ with $v_{1}-v_{2}\in 2\mathbb{Z}$.
\end{proof}
We are almost ready to prove our main result for generic $S_{III}({v_{1},v_{2}})$. Let us now recall what is known so far (some of which were already mentioned in the introduction). Throughout $X_{III}(v_{1},v_{2})$ denotes the set of solutions of $S_{III}(v_{1},v_{2})$.
\begin{fct}{(\cite{NagPil}, \cite{Murata2})}\label{P3Trivial}
Suppose $v_1$ and $v_2$ are algebraically independent over $\mathbb{Q}$. Then $X_{III}(v_{1},v_{2})$ is strongly minimal, geometrically trivial and has no solution in $K^{alg}$.
\end{fct}
Note that by our above discussion, the same holds of $P_{III}(\alpha,\beta,\gamma,\delta)$, where $\alpha,\beta,\gamma,\delta\in\mathbb{C}$ are algebraically independent over $\mathbb{Q}$.
\begin{prop}\label{mainPIII}
Suppose that $v_1$ and $v_2$ are algebraically independent over $\mathbb{Q}$. If $y_{1},...,y_{n}\in X_{III}(v_{1},v_{2})$ are distinct, then $tr.deg_{K}K(y_1,y'_1,\ldots,y_n,y'_n)=2n$.
\end{prop}
\begin{proof}
Let $v_1,v_2\in\mathbb{C}$ be algebraically independent over $\mathbb{Q}$. Throughout, $F$ denotes the field $\mathbb{Q}(v_1,v_2,t)$. Using Fact \ref{P3Trivial}, it suffices to show that if $y_1,y_2\in X_{III}(v_{1},v_{2})$ are distinct, then $tr.deg_{F}F(y_1,y'_1,y_2,y'_2)=4$. So let $y_1$ and $y_2$ be distinct and for contradiction suppose that $y_1\in F\gen{y_2}^{alg}$. By definition, there is an $L_{\partial}$-formula $\phi(x,z,z_1,z_2,z_3)$ such that $\phi(y_1,y_2,v_1,v_2,t)$ holds and witnesses that $y_1$ is algebraic over $F\gen{y_2}$, that is $\exists^{m}x\phi(x,y_2,v_1,v_2,t)$  ($m\in\mathbb{N}$). Note that here both $x,y$ are in the sort $X_{III}(v_{1},v_{2})$. 
\par Since $X_{III}(v_{1},v_{2})$ is strongly minimal and has no solution in $K^{alg}$, we have that $\forall z\exists^{m}x\phi(x,z,v_1,v_2,t)$ is true. Let us denote by  $\theta(z_1,z_2,z_3)$ the formula $\forall z\exists^{m}x\phi(x,z,z_1,z_2,z_3)$. So $\mathcal{U}\models\theta(v_1,v_2,t)$. As $v_1$ and $v_2$ are algebraically independent over $\mathbb{Q}(t)$, for all but finitely many $z_1\in\mathbb{C}$ we have that $\mathcal{U}\models\theta(z_1,v_2,t)$. So in particular for some $\tilde{v}_1\in 2\mathbb{Z}+v_2$, we have that $\theta(\tilde{v}_1,v_2,t)$ holds, that is for all $y\in X_{III}(\tilde{v}_1,v_{2})$ there is a $x\in X_{III}(\tilde{v}_1,v_{2})$ such that $x\in\mathbb{Q}(\tilde{v}_1,v_2,t)\gen{y}^{alg}$. If we take $p\in \mathcal{R}(\tilde{v}_1,v_2)\subset X_{III}(\tilde{v}_1,v_{2})$, then there is $q\in X_{III}(\tilde{v}_1,v_{2})$ such that $q\in\mathbb{Q}(\tilde{v}_1,v_2,t)\gen{p}^{alg}$. But $tr.deg_{\mathbb{Q}(\tilde{v}_1,v_2,t)}\mathbb{Q}(\tilde{v}_1,v_2,t)\gen{p}=1$ and $X_{III}(\tilde{v}_1,v_{2})$ has no solutions in $K^{alg}$. So this forces $q\in \mathcal{R}(\tilde{v}_1,v_2)$. But then $p,q\in \mathcal{R}(\tilde{v}_1,v_2)$ are algebraically dependent over $K$, contradicting Proposition \ref{Secondprop}.
\end{proof}
\subsection{The family $P_{VI}$}

Recall that $P_{VI}(\alpha,\beta,\gamma,\delta)$, $\alpha,\beta,\gamma,\delta \in \mathbb{C}$, is given by the following equation
\begin{equation*}
\begin{split}
y''=&\frac{1}{2}\left(\frac{1}{y}+\frac{1}{y+1}+\frac{1}{y-t}\right)(y')^2-\left(\frac{1}{t}+\frac{1}{t-1}+\frac{1}{y-t}\right)y'\\
&+\frac{y(y-1)(y-t)}{t^2(t-1)^2}\left(\alpha+\beta\frac{t}{y^2}+\gamma\frac{t-1}{(y-1)^2}+\delta\frac{t(t-1)}{(y-t)^2}\right)
\end{split}
\end{equation*}

As with the family $P_{III}$, when working with generic parameters, it is enough to work with the rewriting of the equation as the Hamiltonian system (c.f. \cite{NagPil}):
\[S_{VI}(\boldsymbol{\alpha})\left\{
\begin{array}{rll}
y' &=& \frac{1}{t(t-1)}(2xy(y-1)(y-t)-\{\alpha_4(y-1)(y-t)+\alpha_3y(y-t)\\
& & +(\alpha_0-1)y(y-1)\})\\
x'&=& \frac{1}{t(t-1)}(-x^2(3y^2-2(1+t)y+t)+x\{2(\alpha_0+\alpha_3+\alpha_4-1)y\\
& &-\alpha_4(1+t)-\alpha_3t-\alpha_0+1\}-\alpha_2(\alpha_1+\alpha_2))
\end{array}\right.\]
where $\boldsymbol{\alpha}=(\alpha_{0}, \alpha_{1}, \alpha_{2},\alpha_{3}, \alpha_{4})$ is a tuple of complex numbers such that $\alpha=\frac{1}{2}\alpha_1^2$, $\beta=-\frac{1}{2}\alpha_4^2$, $\gamma=\frac{1}{2}\alpha^2_3$ and $\delta=\frac{1}{2}(1-\alpha^2_0)$ and such that $\alpha_2$ satisfies  $\alpha_0+\alpha_1+2\alpha_2+\alpha_3+\alpha_4=1$.
 As before the Hamiltonian form of $P_{VI}$ reveals some of the family's Riccati subvarieties. The ones we focus on are:
\begin{fct}{(\cite{Lisovyy}, \cite{NagPil})}\\

(1) Suppose that $\alpha_1,\alpha_3,\alpha_4$ are transcendental and algebraically independent over $\mathbb{Q}$, then the Riccati variety $ric(\alpha_1,\alpha_2,\alpha_3,\alpha_4)$ defined by
\[
y=t\;\;\;\;\;\text{ and }\;\;\;\;\;x'= -x^2+\left(\frac{(\alpha_3+\alpha_4-2)t-\alpha_4+1}{t(t-1)}\right)x-\frac{\alpha_2(\alpha_1+\alpha_2)}{t(t-1)}
\]\vspace{0.05 in}
is a differential subvariety of $S_{VI}(0,\alpha_1,\alpha_2,\alpha_3,\alpha_4)$. Furthermore, $ric(\alpha_1,\alpha_2,\alpha_3,\alpha_4)$ has no solution in $K^{alg}$.\\
(2) If $\alpha_1,\alpha_3,\alpha_4\in\mathbb{C}$ are transcendental and algebraically independent over $\mathbb{Q}$ and if $\alpha_0\in 2\mathbb{Z}$,  then there is an order one $\mathbb{C}(t)$-differential subvariety $\mathcal{R}(\boldsymbol{\alpha})$ of $S_{VI}(\boldsymbol{\alpha})$. Furthermore, $\mathcal{R}(\boldsymbol{\alpha})$ has no solutions in $K^{alg}$.

\end{fct}
 As with $S_{III}$, we have that $\mathcal{R}(\boldsymbol{\alpha})=ric(\alpha_1,\alpha_2,\alpha_3,\alpha_4)$ when $\boldsymbol{\alpha}=(0,\alpha_1,\alpha_2,\alpha_3,\alpha_4)$, and that all other $\mathcal{R}(\boldsymbol{\alpha})$'s are obtained from $ric(\alpha_1,\alpha_2,\alpha_3,\alpha_4)$ by applying the well-known Backlund transfromations (c.f. \cite{Okam1}). Furthermore we have the following result (the proof is analogous to that of Proposition \ref{Secondprop})
 
 \begin{prop}\label{Thirdprop}
 Let $\alpha_1,\alpha_3,\alpha_4\in\mathbb{C}$ be transcendental and algebraically independent over $\mathbb{Q}$ and let $\alpha_0\in 2\mathbb{Z}$.  Suppose that $y,z\in \mathcal{R}(\boldsymbol{\alpha})$ be distinct. Then $tr.deg_{K}K(y,z)=2$, that is $y$ and $z$ are algebraically independent over $K=\mathbb{C}(t)$.
 \end{prop}
 
We denote by $X_{VI}(\boldsymbol{\alpha})$ denotes the set of solutions of $S_{VI}(\boldsymbol{\alpha})$.
 
 \begin{fct}{(\cite{NagPil}, \cite{Lisovyy}))}\label{P6Trivial}
 Suppose that $\alpha_0,\alpha_1,\alpha_3,\alpha_4\in\mathbb{C}$ are transcendental and algebraically independent over $\mathbb{Q}$. Then $X_{VI}(\boldsymbol{\alpha})$ is strongly minimal, geometrically trivial and has no solutions in $K^{alg}$.
 \end{fct}
 
 We are now ready for the proof of our main result for the generic sixth Painlev\'e equation.
 
 \begin{prop}\label{mainPVI}
Suppose that $\alpha_0,\alpha_1,\alpha_3,\alpha_4\in\mathbb{C}$ are transcendental and algebraically independent over $\mathbb{Q}$. If $y_{1},...,y_{n}\in X_{VI}(\boldsymbol{\alpha})$ are distinct, then $tr.deg_{K}K(y_1,y'_1,\ldots,y_n,y'_n)=2n$.
\end{prop}
\begin{proof} The proof is very similar to that of Proposition \ref{mainPIII} and so we will leave some details for the reader. Let $\alpha_0,\alpha_1,\alpha_3,\alpha_4\in\mathbb{C}$ be algebraically independent over $\mathbb{Q}$ and let $F$ denotes the field $\mathbb{Q}(\alpha_0,\alpha_1,\alpha_3,\alpha_4,t)$. Using Fact \ref{P6Trivial}, it suffices to show that if $y_1,y_2\in X_{VI}(\boldsymbol{\alpha})$ are distinct, then $tr.deg_{F}F(y_1,y'_1,y_2,y'_2)=4$. So let $y_1$ and $y_2$ be distinct and for contradiction suppose that $y_1\in F\gen{y_2}^{alg}$. As before we have an $L_{\partial}$-formula $\phi(x,z,z_1,z_2,z_3,z_4,z_5)$ such that $\phi(y_1,y_2,\alpha_0,\alpha_1,\alpha_3,\alpha_4,t)$ and $\forall z\exists^{m}x\phi(x,z,\alpha_0,\alpha_1,\alpha_3,\alpha_4,t)$ hold. We denote by $\theta(z_1,z_2,z_3,z_4,z_5)$ the formula $\forall z\exists^{m}x\phi(x,z,z_1,z_2,z_3,z_4,z_5)$. 
\par As $\alpha_0,\alpha_1,\alpha_3,\alpha_4$ are algebraically independent over $\mathbb{Q}(t)$, for all but finitely many $z_1\in\mathbb{C}$ we have that $\mathcal{U}\models\theta(z_1,\alpha_1,\alpha_3,\alpha_4,t)$. In particular for some $\tilde{\alpha}_0\in 2\mathbb{Z}$, we have that $\theta(\tilde{\alpha}_0,\alpha_1,\alpha_3,\alpha_4,t)$ holds. So if we take $p\in \mathcal{R}(\tilde{\alpha}_0,\alpha_1,\alpha_2,\alpha_3,\alpha_4)$, then there is $q\in \mathcal{R}(\tilde{\alpha}_0,\alpha_1,\alpha_2,\alpha_3,\alpha_4)$ such that $q\in\mathbb{Q}(\tilde{\alpha}_0,\alpha_1,\alpha_3,\alpha_4,t)\gen{p}^{alg}$ (here we explicitly write $\mathcal{R}(\tilde{\alpha}_0,\alpha_1,\alpha_2,\alpha_3,\alpha_4)$ instead of $\mathcal{R}(\boldsymbol{\alpha})$). But this contradicts Proposition \ref{Thirdprop}.
\end{proof}


\end{document}